\documentclass[12pt,a4paper]{article}%
\usepackage[utf8]{inputenc}
\usepackage{hyperref}
\usepackage{amsmath}
\usepackage{amsfonts}
\usepackage{amssymb}
\usepackage{xcolor}
\usepackage[space]{grffile}
\usepackage{graphicx}%
\providecommand{\U}[1]{\protect\rule{.1in}{.1in}}
\newtheorem{theorem}{Theorem}

\newtheorem{conjecture}[theorem]{Conjecture}
\newtheorem{corollary}[theorem]{Corollary}

\newtheorem{lemma}[theorem]{Lemma}

\newtheorem{problem}[theorem]{Problem}
\newtheorem{proposition}[theorem]{Proposition}

\newtheorem{observation}[theorem]{Observation}

\newenvironment{proof}[1][Proof]{\noindent\textbf{#1.} }{\ \hfill \rule{0.5em}{0.5em}\bigskip}
\graphicspath{{D:/Dropbox/Riste-Sedlar/Yu Yang/Paper1/Slike/}}

\begin{document}

\title{Invitation to the subpath number}
\author{Martin Knor$^{1}$, Jelena Sedlar$^{2,4}$, Riste \v{S}krekovski$^{3,4,5}$, Yu
Yang$^{6}$\\{\small $^{1}$ \textit{Slovak University of Technology, Bratislava, Slovakia
}}\\[0.1cm] {\small $^{2}$ \textit{University of Split, FGAG, Split, Croatia }}\\[0.1cm] {\small $^{3}$ \textit{University of Ljubljana, FMF, Ljubljana,
Slovenia }}\\[0.1cm] {\small $^{4}$ \textit{Faculty of Information Studies, Novo Mesto,
Slovenia }}\\[0.1cm] {\small $^{5}$ \textit{Rudolfovo - Science and Technology Centre, Novo
mesto}\textit{, Slovenia }}\\[0.1cm] {\small $^{6}$ \textit{School of Software, Pingdingshan University,
Pingdingshan, China}}}
\maketitle

\begin{abstract}
In this paper we count all the subpaths of a given graph $G,$ including the
subpaths of length zero, and we call this quantity the subpath number of $G$.
The subpath number is related to the extensively studied number of subtrees,
as it can be considered as counting subtrees with the additional requirement
of maximum degree being two. We first give the explicit formula for the
subpath number of trees and unicyclic graphs. We show that among connected
graphs on the same number of vertices, the minimum of the subpath number is
attained for any tree and the maximum for the complete graph. Further, we show
that the complete bipartite graph with partite sets of almost equal size
maximizes the subpath number among all bipartite graphs. The explicit formula
for cycle chains, i.e. graphs in which two consecutive cycles share a single
edge, is also given. This family of graphs includes the unbranched
catacondensed benzenoids which implies a possible application of the result in
chemistry. The paper is concluded with several directions for possible further
research where several conjectures are provided.

\end{abstract}

\section{Introduction}

The number of non-empty subtrees $N(G)$ in a graph $G$ is a concept that was
first studied for a graph $G$ which is a tree \cite{Szekely2005}. Various
properties of $N(G)$ have been studied for several subclasses of trees, such
as trees with given maximum degree \cite{Kirk2008}, binary trees
\cite{Szekely2007}, trees with given degree sequence \cite{Zhang2013}.
Recently, the study of the number of subtrees has been extended to some
classes of general graphs, such as graphs with given number of cut edges
\cite{Xu2021} and cacti and block graphs \cite{Cacti2022}. The inverse problem
for the number of subtrees was studied in \cite{Czabarka2008}. For some other
interesting results on the topic we refer the reader to
\cite{YuYang1,YuYang3,YuYangHexagonal}.

Motivated by this research, in this paper we study the number of subpaths in a
graph $G$, we call this quantity the subpath number. The number of subtrees
has an enormous magnitude due to which it is hard to evaluate, but we believe
that the number of subpaths should be more manageable. Nevertheless, the
complexity of counting subpaths of a graph has been studied in
\cite{Yamamoto2017}, where it is established that this problem is $\#P$-hard.

Most research regarding the subtree number is done on the subtrees of trees.
In the case of counting subpaths, every pair of vertices in a tree is
connected by only one path, so the subpath number is the same for all trees on
a given number of vertices. Hence, it is interesting to consider the subpath
number only for classes of general graphs. We establish the exact value of the
subpath number for unicyclic graphs, which yields extremal graphs in that
class. By showing that the subpath number is strictly increasing with respect
to the edge addition, we show that the complete graph maximizes the subpath
number, and that trees minimize it among all connected graphs on a given
number of vertices. In the class of bipartite graphs, we show that the
complete bipartite graph with partite sets of (almost) equal size maximizes
the subpath number. We also consider graphs consisting of a sequences of
cycles in which two consecutive cycles share a single edge. For such graphs we
establish the exact formula for the subpath number. Since this class of graphs
include ladder graphs, but also the graphs of the unbranched catacondesed
benzenoids, our results might have a potential application to chemistry
\cite{Anstoter2020}.

The Wiener index $W(G)$ of a graph $G$ is defined as the sum of distances over
all pairs of vertices in $G$. It was introduced in the seminal paper by Wiener
\cite{Wiener}, where a correlation of the Wiener index with the chemical
properties of some molecular compounds is established. Since then, the Wiener
index has become one of the most researched indices in chemical graph theory,
for an overview of the results we refer the reader to two surveys
\cite{risteSurvey1, risteSurvey2}. An interesting "negative" correlation of
the number of subtrees and the Wiener index is also observed, in the sense
that on many graph classes the graph which maximizes the number of subtrees is
the same graph which minimizes the Wiener index, and vice versa. This is
observed on several graph classes such as trees, unicyclic graphs and cacti,
for more details see \cite{Xu2022}. One possible direction for further
research is to establish whether the subpath number correlates with one of
these quantities in this sense, or a class of graphs on which it does not
correlate to either of them.

This paper is concluded with some other considerations regarding further work.
Given the result for bipartite graphs, we conjecture that the complete
bipartite graph with partite sets of (almost) equal size maximizes the subpath
number among all triangle-free graphs also. Moreover, since the addition of an
edge strictly increases the subpath number, an interesting question is whether
a regular graph maximizes (resp. minimizes) the subpath number among all
graphs with prescribed maximum degree (resp. minimum degree). This focuses
one's attention to regular graphs and the question which graphs are maximal
and which are minimal with respect to the subpath number in the class of
$r$-regular graph. Since a $2$-regular graph is a collection of cycles for
which the subpath number is easily established, the first non-trivial case is
$r=3$, i.e. cubic graphs. For this case, we conjecture extremal graphs.

\section{Preliminary results}

For a graph $G$, the \emph{subpath number} is defined as the number of paths
in a graph, including the trivial paths of length $0$. The subpath number of a
graph $G$ is denoted by $\mathrm{pn}(G).$ Let us first observe some
interesting properties of such a number. Assume that the vertices of $G$ are
denoted by $v_{1},\ldots,v_{n},$ where $d_{i}$ denotes the degree of a vertex
$v_{i}.$ We can now partition the paths of the graph $G$ into classes of given
length, i.e., we define ${\mathrm{pn}}_{l}(G)$ to be the number of all paths
of length $l$ in $G.$ Obviously, it holds that
\[
{\mathrm{pn}}(G)=\sum_{l=0}^{n-1}{\mathrm{pn}}_{l}(G).
\]
Before we proceed to the proposition which gives the value of $\mathrm{pn}%
_{l}(G)$ for small $l,$ let us first introduce the first and the second Zagreb
index of a graph $G,$ denoted by $M_{1}(G)$ and $M_{2}(G).$ These two indices
are defined by%
\[
M_{1}(G)=\sum_{i=1}^{n}d_{i}^{2}\text{ \ \ and \ \ }M_{2}(G)=\sum
_{e=v_{i}v_{j}\in E(G)}d_{i}d_{j},
\]
respectively. Also, a cycle of length $3$ will be called a \emph{triangle}.
For the paths of a small length the following proposition holds.

\begin{proposition}
Let $G$ be a graph on $n$ vertices denoted by $v_{i},$ for $i=1,\ldots,n,$
where $d_{i}$ denotes the degree of the vertex $v_{i}.$ Then the following holds:

\begin{itemize}
\item ${\mathrm{pn}}_{0}(G)=\left\vert V(G)\right\vert ;$

\item ${\mathrm{pn}}_{1}(G)=\left\vert E(G)\right\vert ;$

\item ${\mathrm{pn}}_{2}(G)=\sum_{i=1}^{n}\binom{d_{i}}{2};$

\item ${\mathrm{pn}}_{3}(G)=M_{2}(G)-M_{1}(G)+\left\vert E(G)\right\vert -3t,$
where $t$ denotes the number of triangles in $G.$
\end{itemize}
\end{proposition}

\begin{proof}
Since every vertex of $G$ represents one path of length $0$ in $G,$ it follows
that ${\mathrm{pn}}_{0}(G)=\left\vert V(G)\right\vert .$ Similarly, since each
edge of $G$ represents one path of length $1$ in $G,$ it follows that
${\mathrm{pn}}_{1}(G)=\left\vert E(G)\right\vert $. Let us next consider
${\mathrm{pn}}_{2}(G)$. Denote by $\mathcal{P}_{i}$ the set of all paths of
length two which have $v_{i}$ for the middle vertex. Both edges of such a path
are incident to $v_{i},$ hence $\left\vert \mathcal{P}_{i}\right\vert
=\binom{d_{i}}{2}.$ The formula for ${\mathrm{pn}}_{2}(G)$ now follows from
the fact that the sets $\mathcal{P}_{i}$ and $\mathcal{P}_{j}$ are disjoint
for $i\not =j.$

Let us finally consider ${\mathrm{pn}}_{3}(G)$. Here, we focus on the middle
edge $e$ of a path in $G$ of length $3$. Denote by $\mathcal{P}_{i,j}$ the set
of all paths of $G$ which have $e=v_{i}v_{j}$ as the middle edge. Notice that
the other edge of such a path incident to $v_{i}$ can be chosen in $d_{i}-1$
ways and the other edge incident to $v_{j}$ in $d_{j}-1$ ways. Thus, we obtain
quantity $(d_{i}-1)(d_{j}-1)$, but here we also counted the triangles of $G$
which contain $e$ as an edge. Since two paths of length three with distinct
middle edge must be distinct, we obtain
\begin{align*}
\mathrm{pn}_{3}(G)  &  =\sum_{e=v_{i}v_{j}\in E(G)}(d_{i}-1)(d_{j}-1)-3t\\
&  =\sum_{e=v_{i}v_{j}\in E(G)}d_{i}d_{j}-\sum_{e=v_{i}v_{j}\in E(G)}%
(d_{i}+d_{j})+\left\vert E(G)\right\vert -3t.
\end{align*}
The first member of the above sum equals $M_{2}(G)$ by definition, so it
remains to show that the second member of the sum equals $M_{1}(G).$ To see
this, notice that for each vertex $v_{i}$ of $G,$ the degree $d_{i}$ is added
in the sum $d_{i}$ times, hence we have%
\[
\sum_{e=v_{i}v_{j}\in E(G)}(d_{i}+d_{j})=\sum_{v_{i}\in V(G)}d_{i}^{2}%
=M_{1}(G),
\]
and we are done.
\end{proof}

Let us next consider trees on $n$ vertices. From the well known fact that
every pair of vertices in a tree is connected by the unique path, the
following observation follows.

\begin{observation}
\label{Obs_tree}If $T$ is a tree on $n$ vertices, then ${\mathrm{pn}%
}(T)=\binom{n+1}{2}$.
\end{observation}

Since unicyclic graphs are obtained from trees by introducing a single edge,
the natural next step is to consider unicyclic graphs on $n$ vertices. In a
unicyclic graph $G$ with the cycle of length $g,$ removing the edges of the
cycle results in precisely $g$ connected components.

\begin{proposition}
\label{Prop_unicyclic} Let $G$ be a unicyclic graph on $n$ vertices with the
cycle $C$ of the length $g$. Denote by $n_{1},n_{2},\dots,n_{g}$ the number of
vertices in the components which occur after deleting all edges of $C$. Then,
${\mathrm{pn}}(G)=n+2\binom{n}{2}-\binom{n_{1}}{2}-\binom{n_{2}}{2}%
-\dots-\binom{n_{g}}{2}$.
\end{proposition}

\begin{proof}
Denote by $C=v_{1}v_{2}\cdots v_{g}v_{1}$ the only cycle in $G.$ Let $G_{i}$
be the connected component of $G\backslash E(C)$ which contains $v_{i}.$ We
may assume that $G_{i}$ has $n_{i}$ vertices. Notice the following, if a pair
of vertices belongs to the same connected component $G_{i}$, then there exists
a unique path connecting them. On the other hand, if two vertices belong to
two distinct connected components of $G\backslash E(C),$ say $v_{i}$ and
$v_{j},$ then there are precisely two distinct paths in $G$ connecting such a
pair of vertices. So, we count the paths of $G$ as follows. First, there are
$n$ paths of length zero, as each vertex of $G$ represents one such paths.
Second, for each pair of vertices of $G$ we may count that there are two paths
which connect them, this gives $2\binom{n}{2}$ paths. From this quantity we
must subtract the number of pairs of vertices in $G$ which are connected by
only one path. For each connected component $G_{i}$ there are $\binom{n_{i}%
}{2}$ such pairs. Hence, we obtain%
\[
{\mathrm{pn}}(G)=n+2\binom{n}{2}-\binom{n_{1}}{2}-\binom{n_{2}}{2}%
-\dots-\binom{n_{g}}{2},
\]
as claimed.
\end{proof}

The above result easily yields the value of the subpath number for the cycle
on $n$ vertices.

\begin{corollary}
\label{Cor_cycle} For the cycle on $n$ vertices we have ${\mathrm{pn}}%
(C_{n})=n^{2}$.
\end{corollary}

Also, from the formula for the subpath number of a unicyclic graph given in
Proposition \ref{Prop_unicyclic}, it is easy to deduce extremal unicyclic
graphs with respect to the subpath number.

\begin{corollary}
\label{Cor_unicyclic} Among unicyclic graphs on $n$ vertices, ${\mathrm{pn}%
}(G)$ attains maximum value if and only if $G=C_{n}$, and it attains the
minimum value if and only if the only cycle of $G$ is a triangle with two of
its vertices being of degree two.
\end{corollary}

\bigskip

Let us next consider a random graph $G(n,p)$ on $n$ vertices where the
probability of each edge is $p,$ for $0\leq p\leq1.$ The following proposition
gives the expected value of the subpath number for such a graph.

\begin{proposition}
\label{Prop_random}For a random graph $G(n,p),$ where $0\leq p\leq1,$ it holds
that
\[
\mathbb{E}(\mathrm{pn}(G(n,p)))=\frac{1}{2}\sum_{k=1}^{n}\frac{n!}%
{(n-k)!}\text{ }p^{k-1}+\frac{n}{2}.
\]

\end{proposition}

\begin{proof}
A sequence of $k$ vertices $P=v_{1}v_{2}\cdots v_{k}$ ($k\geq2$) induces a
path in $G(n,p)$ with probability $p^{k-1}$. There are $n!/(n-k)!$ such
sequences, so the expectation of the number of $k$-paths $X_{k}$ in $G(n,p)$
is $\mathbb{E}(X_{k})=1/2$ $n!/(n-k)!$ $p^{k-1}$, wherein factor $1/2$ appears
since each path corresponds to two such sequences. As there are always $n$
paths of length $0,$ we obtain
\[
\mathbb{E}(\mathrm{pn}(G(n,p)))=%
{\displaystyle\sum_{k=1}^{n}}
\mathbb{E}(X_{k})=\frac{1}{2}\sum_{k=1}^{n}\frac{n!}{(n-k)!}\text{ }%
p^{k-1}+\frac{n}{2},
\]
as claimed.
\end{proof}

The above proposition yields the value of the subpath number for the complete
graph $K_{n},$ since the probability $p$ of each edge in $K_{n}$ equals $1.$
It is easily verified that
\[
\sum_{k=1}^{n}\binom{n}{k}k!=n!\sum_{k=0}^{n-1}\frac{1}{k!},
\]
so we have the following result.

\begin{corollary}
\label{Prop_Kn} For the complete graph $K_{n},$ we have
\[
{\mathrm{pn}}(K_{n})=\frac{n!}{2}\sum_{k=0}^{n-1}\frac{1}{k!}+\frac{n}{2}.
\]

\end{corollary}

\noindent Approximating further the expression for ${\mathrm{pn}}(K_{n}),$ we
obtain%
\[
{\mathrm{pn}}(K_{n})\sim\frac{n!}{2}e+\frac{n}{2}.
\]
Also, we establish the value of the subpath number for the complete bipartite graph.

\begin{proposition}
\label{Prop_Kab}For the complete bipartite graph $K_{a,b},$ where $a+b=n$ and
$a\leq b,$ we have%
\begin{align*}
{\mathrm{pn}}(K_{a,b})={}  &  \binom{a}{2}%
{\displaystyle\sum_{k=1}^{a}}
\binom{b}{k}k!\cdot\binom{a-2}{k-1}(k-1)!\\
&  +\binom{b}{2}%
{\displaystyle\sum_{k=1}^{a+1}}
\binom{a}{k}k!\cdot\binom{b-2}{k-1}(k-1)!\\
&  +a\cdot b%
{\displaystyle\sum_{k=1}^{a}}
\binom{a-1}{k-1}(k-1)!\cdot\binom{b-1}{k-1}(k-1)!+n.
\end{align*}

\end{proposition}

\begin{proof}
Denote the two bipartite sets of $K_{a,b}$ by $A$ and $B,$ so that $\left\vert
A\right\vert =a$ and $\left\vert B\right\vert =b$. Assume first that $x,y\in
A$ and $x\not =y$. Notice that there are $\binom{a}{2}$ such pairs. Since
$K_{a,b}$ is bipartite, $x$ and $y$ can be connected only by a path of an even
length. The number of subpaths of $G$ of the length $2k$ which connect $x$ and
$y$ is $\binom{b}{k}k!\cdot\binom{a-2}{k-1}(k-1)!.$ Since $d(x,y)\leq2a,$ it
follows that $1\leq k\leq a,$ so there are
\[
\binom{a}{2}%
{\displaystyle\sum_{k=1}^{a}}
\binom{b}{k}k!\cdot\binom{a}{k-1}(k-1)!
\]
paths which connect such pairs $x$ and $y.$ Notice that this is the expression
from the first line of the formula for ${\mathrm{pn}}(K_{a,b})$ from the
statement. The same reasoning for $x,y\in B$ and $x\not =y$, yields the
expression from the second line of the formula. Notice that due to $a\leq b,$
here we have $1\leq k\leq a+1$ and not $1\leq k\leq b.$

Assume next that $x\in A$ and $y\in B.$ There are $a\cdot b$ such pairs and
obviously $x\not =y.$ Vertices $x$ and $y$ can be connected only by a path of
an odd length $2k-1$ for $1\leq k\leq a.$ For such a pair $x$ and $y$, there
are $\binom{a-1}{k-1}(k-1)!\cdot\binom{b-1}{k-1}(k-1)!$ subpaths of $K_{a,b}$
connecting them. Finally, the $n$ subpaths of zero length must be added to
${\mathrm{pn}}(K_{n})$ and we are done.
\end{proof}

Next, we consider all connected graphs where we show that the minimum value of
the subpath number is obtained for any tree and the maximum value for the
complete graph. For that purpose we need the following statement.

\begin{lemma}
\label{Lemma_aditivity}Let $G$ be a connected graph on $n$ vertices and let
$e$ be an edge of $G$. Let $G^{\prime}$ be a graph obtained from $G$ by
removing the edge $e$. Then ${\mathrm{pn}}(G^{\prime})<{\mathrm{pn}}(G)$.
\end{lemma}

\begin{proof}
The claim follows from the fact that every path in $G^{\prime}$ is also a path
in $G.$ On the other hand, an edge $e$ is the path of the length $1$ in $G,$
which is not contained in $G^{\prime}.$
\end{proof}

We now have the following result.

\begin{theorem}
\label{Tm_generalG}Let $G$ be a connected graph on $n$ vertices. Then
\[
\binom{n}{2}\leq{\mathrm{pn}}(G)\leq\frac{n!}{2}\sum_{i=0}^{n-1}\frac{1}%
{i!}+\frac{n}{2},
\]
where the lower bound is attained if and only if $G$ is a tree, and the upper
bound if and only if $G=K_{n}.$
\end{theorem}

\begin{proof}
The equality in lower and upper bound follows from Observation~{\ref{Obs_tree}%
} and Corollary~{\ref{Prop_Kn}}, respectively. Let us now prove the
inequalities. Let $G$ be a connected graph on $n$ vertices which is neither a
tree, nor a complete graph. Then, there exists a sequence of edges of $G$ by
removing which a tree is obtained. Also, there exist a sequence of edges which
can be added to $G$ to obtain the complete graph $K_{n}.$ Then the claim
follows from Lemma \ref{Lemma_aditivity}.
\end{proof}

\section{Bipartite graphs}

In the previous section we established the subpath number of the complete
bipartite graphs.\ In this section we consider bipartite graphs more
generally, namely we characterize extremal bipartite graphs with respect to
the subpath number. Since every tree is a bipartite graph, Theorem
\ref{Tm_generalG} implies that a bipartite graph $G$ minimizes the subpath
number if and only if $G$ is a tree. It remains to consider the bipartite
graphs which maximize the subpath number.

\begin{figure}[h]
\begin{center}%
\begin{tabular}
[t]{ll}%
a) & \raisebox{-0.9\height}{\includegraphics[scale=0.6]{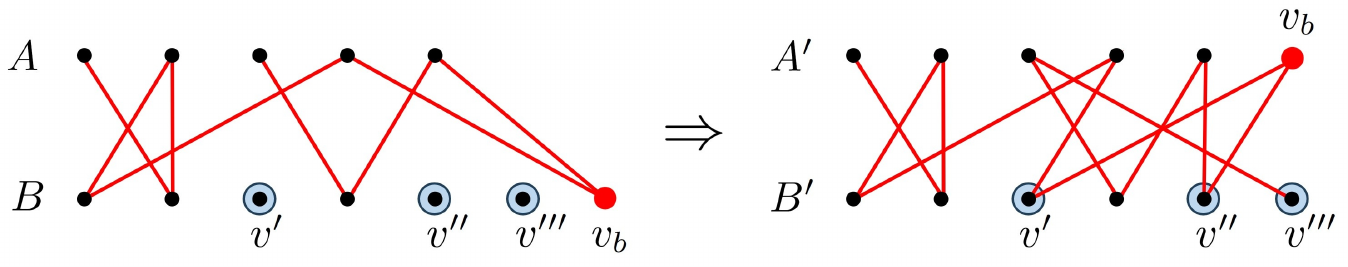}}\\
b) & \raisebox{-0.9\height}{\includegraphics[scale=0.6]{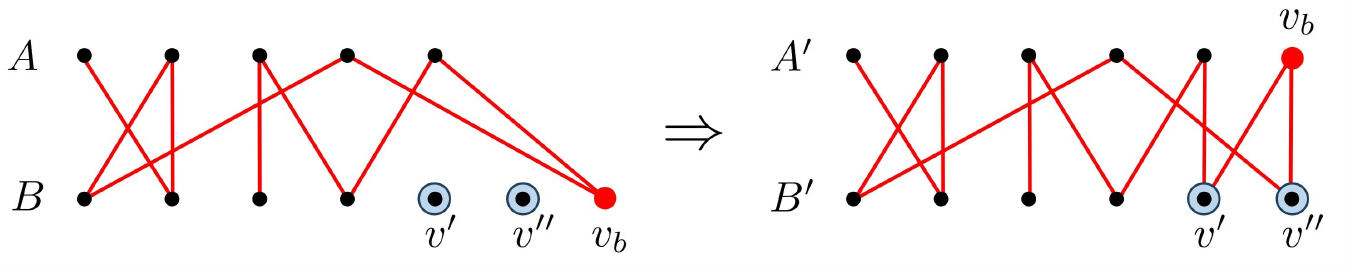}}\\
c) & \raisebox{-0.9\height}{\includegraphics[scale=0.6]{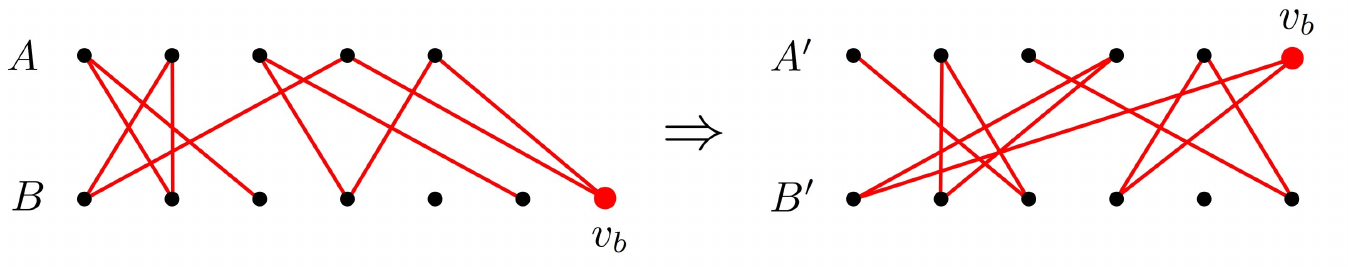}}
\end{tabular}
\end{center}
\caption{All figures on the left show a path $P$ in the complete bipartite
graph $K_{5,7}$ which contains $v_{b}=v_{7}.$ The figures on the right show
the corresponding path $f(P)=P^{\prime}$ in $K_{6,6}$. Edges of $K_{5,7}$ and
$K_{6,6}$ not contained on the path are not shown for the simplicity sake.
Notice that: a) $P=u_{1}v_{2}u_{2}v_{1}u_{4}\protect\frame{$v_7$}u_{5}%
v_{4}u_{3}\in\mathcal{P}_{A}$ so $P^{\prime}=u_{1}v_{2}u_{2}v_{1}u_{4}%
v_{3}\protect\frame{$v_7$}v_{5}u_{5}v_{4}u_{3}\in\mathcal{P}_{A,B}^{\prime},$
b) $P=u_{1}v_{2}u_{2}v_{1}u_{4}\protect\frame{$v_7$}u_{5}v_{4}u_{3}v_{3}%
\in\mathcal{P}_{A,B}$ so $P^{\prime}=u_{1}v_{2}u_{2}v_{1}u_{4}v_{6}%
\protect\frame{$v_7$}v_{5}u_{5}v_{4}u_{3}v_{3}\in\mathcal{P}_{A,B}^{\prime},$
c) $P=v_{3}u_{1}v_{2}u_{2}v_{1}u_{4}\protect\frame{$v_7$}u_{5}v_{4}u_{3}%
v_{6}\in\mathcal{P}_{B}$ so $P^{\prime}=u_{1}v_{3}u_{2}v_{2}u_{4}%
v_{1}\protect\frame{$v_7$}v_{4}u_{5}v_{6}u_{3}\in\mathcal{P}_{A}.$}%
\label{Fig_bipartite}%
\end{figure}

\begin{theorem}
\label{Tm_bipartite}Among bipartite graphs on $n$ vertices the maximum value
of the subpath number is attained only by the complete bipartite graph
$K_{\left\lceil n/2\right\rceil ,\left\lfloor n/2\right\rfloor }.$
\end{theorem}

\begin{proof}
Let $G$ be a bipartite graph with partition sets of sizes $a$ and $b,$ where
we assume $a\leq b.$ If $G$ is not a complete bipartite graph, then Lemma
\ref{Lemma_aditivity} implies $\mathrm{pn}(G)<\mathrm{pn}(K_{a,b}).$ Thus,
assume that $G$ is a complete bipartite graph, i.e., $G=K_{a,b}$. If
$b-a\leq1,$ then $G=K_{\left\lceil n/2\right\rceil ,\left\lfloor
n/2\right\rfloor }$, and the claim follows. Otherwise, assume $b-a\geq2.$ Let
$A=\{u_{1},\ldots,u_{a}\}$ be the smaller bipartite set of $G$ and
$B=\{v_{1},\ldots,v_{b}\}$ the larger one. We denote $u_{i}\prec u_{j}$ (resp.
$v_{i}\prec v_{j}$) if $i<j.$ Define $G^{\prime}$ as the graph obtained from
$G$ by removing the edges $v_{b}u_{i}$ for $1\leq i\leq a$ and introducing the
edges $v_{b}v_{i}$ for $1\leq i\leq b-1.$ Notice that $G^{\prime}%
=K_{a+1,b-1},$ where the two partite sets are $A^{\prime}=A\cup\{v_{b}\}$ and
$B^{\prime}=B\backslash\{v_{b}\}.$ As for the labels and the corresponding
ordering of vertices of $A^{\prime}$ and $B^{\prime},$ they are inherited from
$G,$ where we additionally have to define only $u_{i}\prec v_{b}$ for every
$i\leq a.$

We now show that $\mathrm{pn}(G)<\mathrm{pn}(G^{\prime}).$ Let $\mathcal{P}$
and $\mathcal{P}^{\prime}$ denote the sets of all subpaths of $G$ and
$G^{\prime}$, respectively. Our goal is to construct an injective but not
surjective mapping $f:\mathcal{P}\rightarrow\mathcal{P}^{\prime},$ as the
existence of such $f$ implies $\mathrm{pn}(G)<\mathrm{pn}(G^{\prime}).$

First, we define partitions of $\mathcal{P}$ and $\mathcal{P}^{\prime}$ which
will be useful to us. A path $P\in\mathcal{P}$ belongs to $\mathcal{P}_{A}$
(resp. $\mathcal{P}_{B}$) if both its end-vertices belong to $A$ (resp. $B$).
Similarly, a path $P^{\prime}\in\mathcal{P}^{\prime}$ belongs to
$\mathcal{P}_{A}^{\prime}$ (resp. $\mathcal{P}_{B}^{\prime}$) if both its
end-vertices belong to $A^{\prime}$ (resp. $B^{\prime}$). We also define
$\mathcal{P}_{AB}=\mathcal{P}\backslash(\mathcal{P}_{A}\cup\mathcal{P}_{B})$
and we obtain the partition of $\mathcal{P}$ into $\mathcal{P}_{A},$
$\mathcal{P}_{B}$ and $\mathcal{P}_{AB}.$ Similarly we define $\mathcal{P}%
_{AB}^{\prime}$ to obtain the analogous partition of $\mathcal{P}^{\prime}.$

The mapping $f$ is defined as follows. If a path $P\in\mathcal{P}$ does not
contain $v_{b},$ set $f(P)=P.$ Now, consider a subpath $P$ of $G$ that
contains $v_{b}.$ Denote vertices of $P$ by $x_{1}x_{2}\ldots x_{k}$ so that
$x_{i}x_{i+1}$ is an edge of $P.$ If $P\in\mathcal{P}_{A}\cup\mathcal{P}%
_{AB},$ we assume $x_{1}\in A.$ Additionally, if $P\in\mathcal{P}_{A}%
\cup\mathcal{P}_{B},$ we assume $x_{1}\prec x_{k}.$ Denote by $i^{\ast}$ the
index for which $x_{i^{\ast}}=v_{b}.$ We construct the corresponding subpath
$P^{\prime}$ in the following way illustrated by Figure \ref{Fig_bipartite}:

\begin{itemize}
\item If $P\in\mathcal{P}_{A},$ then $b-a\geq2$ implies $B\backslash V(P)$
contains at least three vertices. Denote them by $v^{\prime},$ $v^{\prime
\prime}$ and $v^{\prime\prime\prime}$ so that $v^{\prime}\prec v^{\prime
\prime}\prec v^{\prime\prime\prime}.$ Let $P^{\prime}$ be the path obtained
from $P$ by inserting $v^{\prime}$ right before $x_{i^{\ast}}$, $v^{\prime
\prime}$ right after $x_{i^{\ast}}$ and $v^{\prime\prime\prime}$ after $x_{k}$
in the vertex sequence of $P$. Notice that $P^{\prime}$ is a subpath of
$G^{\prime}$ and $P^{\prime}\in\mathcal{P}_{AB}^{\prime}.$

\item If $P\in\mathcal{P}_{AB},$ then $b-a\geq2$ implies $B\backslash V(P)$
contains at least two vertices. Denote them by $v^{\prime}$ and $v^{\prime
\prime}$ so that $v^{\prime}\prec v^{\prime\prime}$. Let $P^{\prime}$ be a
path obtained from $P$ by inserting $v^{\prime\prime}$ right before
$x_{i^{\ast}}$ and $v^{\prime}$ right after $x_{i^{\ast}}$. Then again,
$P^{\prime}$ is a subpath of $G^{\prime}$ with $P^{\prime}\in\mathcal{P}%
_{AB}^{\prime}.$

\item If $P\in\mathcal{P}_{B},$ then $i^{\ast}$ is odd and $k$ is odd. Let
$P^{\prime}$ be the path obtained from $P$ by swapping $x_{2i-1}$ and $x_{2i}%
$\ for $2i<i^{\ast},$ and also swapping $x_{2i}$ and $x_{2i+1}$ for
$2i>i^{\ast}.$ Again, $P^{\prime}$ is a subpath of $G^{\prime}$ and
$P^{\prime}\in\mathcal{P}_{A}^{\prime}.$
\end{itemize}

\noindent We define $f(P)=P^{\prime}$, so the mapping $f$ is completely
defined. Notice that $f(P)$ contains $v_{b}$ if and only if $P$ contains
$v_{b}.$

Let us first show that $f$ is injective, i.e. that $f(P)=f(Q)$ implies $P=Q$.
For that purpose, let $P,Q\in\mathcal{P}$ be a pair of paths with $f(P)=f(Q).$
We denote $P^{\prime}=f(P)$ and $Q^{\prime}=f(Q)$. If $P^{\prime}=Q^{\prime}$
does not contain $v_{b}$, then $P=P^{\prime}=Q^{\prime}=Q$. So, let us assume
that $P^{\prime}=Q^{\prime}$ contains $v_{b}$. Let $P=x_{1}\cdots x_{k}$ and
$Q=y_{1}\cdots y_{q}$ be the vertex labelings of $P$ and $Q.$ Similarly,
assume $P^{\prime}=x_{1}^{\prime}\cdots x_{k^{\prime}}^{\prime}$ and
$P^{\prime}=Q^{\prime}$ implies this is also a labeling of $Q^{\prime}.$
Analogously as in $\mathcal{P}$, we assume that if $P^{\prime}\in
\mathcal{P}_{A}^{\prime}\cup\mathcal{P}_{AB}^{\prime}$, then $x_{1}^{\prime
}\in A$. Additionally, if $P^{\prime}\in\mathcal{P}_{A}^{\prime}%
\cup\mathcal{P}_{B}^{\prime}$, then we assume $x_{1}^{\prime}\prec
x_{k}^{\prime}$. Let $i^{\ast\ast}$ satisfies $x_{i^{\ast\ast}}^{\prime}%
=v_{b}$. We distinguish three cases.

\medskip

\noindent\textbf{Case 1:} $P^{\prime}\in\mathcal{P}_{A,B}^{\prime}$ and
$x_{i^{\ast\ast}-1}^{\prime}\prec x_{i^{\ast\ast}+1}^{\prime}$. Then both
$P^{\prime}$ and $Q^{\prime}$ were obtained from paths in $\mathcal{P}_{A}$,
i.e. $P,Q\in\mathcal{P}_{A}$ Hence, both $P$ and $Q$ are obtained from
$P^{\prime}$ by deleting $x_{i^{\ast\ast}-1}^{\prime},x_{i^{\ast\ast}%
+1}^{\prime},x_{k^{\prime}}^{\prime}$ which implies $P=Q$.

\medskip

\noindent\textbf{Case 2:} $P^{\prime}\in\mathcal{P}_{A,B}^{\prime}$ and
$x_{i^{\ast\ast}+1}^{\prime}\prec x_{i^{\ast\ast}-1}^{\prime}$. Then both
$P^{\prime}$ and $Q^{\prime}$ were obtained from paths in $\mathcal{P}_{A,B}$.
Hence, $P$ as well as $Q$ are obtained from $P^{\prime}$ by deleting
$x_{i^{\ast\ast}-1}^{\prime},x_{i^{\ast\ast}+1}^{\prime}$, so $P=Q$.

\medskip

\noindent\textbf{Case 3:} $P^{\prime}\in\mathcal{P}_{A}^{\prime}$. Then both
$P^{\prime}$ and $Q^{\prime}$ were obtained from paths in $\mathcal{P}_{B}$.
Hence, both $P$ and $Q$ are obtained from $P^{\prime}$ by swapping
$x_{2i-1}^{\prime}$ and $x_{2i}^{\prime}$ for $2i<i^{\ast\ast}$, and also by
swapping $x_{2i}^{\prime}$ and $x_{2i+1}^{\prime}$ for $2i>i^{\ast\ast}$. So
again $P=Q$, although we have to reverse both $P$ and $Q$ if $x_{k^{\prime}%
-1}^{\prime}\prec x_{2}^{\prime}$.

\medskip

Therefore, $f$ is an injection. It is easily seen that $f$ is not a surjection
since the path $v_{1}v_{b}$ of $G^{\prime}$ is not $f(P)$ for any path $P$ of
$G$.
\end{proof}

From the above theorem and Proposition \ref{Prop_Kab}, the sharp upper bound
on $\mathrm{pn}(G)$ for a bipartite graph $G$ can be derived.

\section{Cycle chains}

As the next class of graphs we consider the so called cycle chains where two
consecutive cycles share a single edge and no vertex is shared by more than 2
cycles. More precisely, let $k\geq2$ and let $S=(a_{1},a_{2},\dots,a_{k}%
;b_{1},b_{2},\dots,b_{k})$ be a sequence of positive integers such that
$a_{1}=b_{1}$ and $a_{k}=b_{k}$. Take a collection of $k-1$ independent edges
$u_{1}v_{1},u_{2}v_{2},\dots,u_{k-1}v_{k-1}$. Now join $u_{1}$ with $v_{1}$ by
a path of length $a_{1}+1$ and join $u_{k-1}$ with $v_{k-1}$ by a path of
length $a_{k}+1$. Further, for $2\leq i\leq k-1$, join $u_{i-1}$ with $u_{i}$
by a path of length $a_{i}$ and join $v_{i-1}$ with $v_{i}$ by a path of
length $b_{i}$. Finally, denote the resulting graph by $G(S)$ and called a
\emph{cycle chain}. A graph $G(S)$ is illustrated by Figure \ref{FigMartin}.

\begin{figure}[h]
\begin{center}
\includegraphics[scale=0.6]{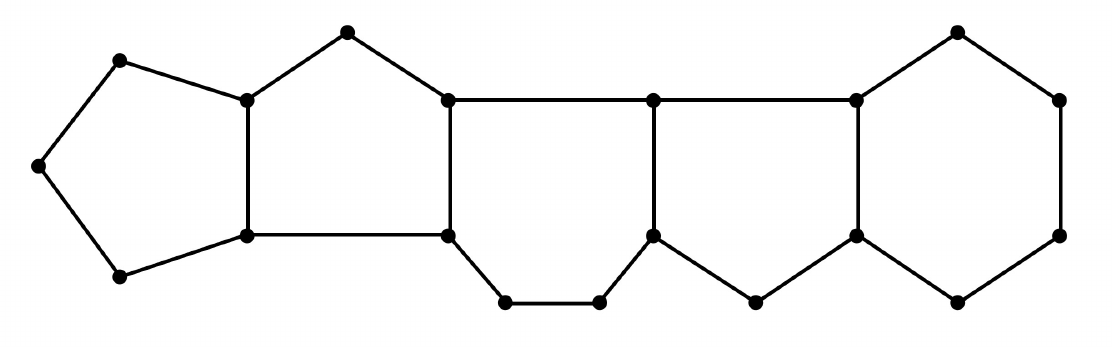}
\end{center}
\caption{Graph $G(S)$ with $S=(3,2,1,1,4;3,1,3,2,4)$.}%
\label{FigMartin}%
\end{figure}

Notice that $G(S)$ is a graph which consists of the sequence of $k$ cycles
$C_{i},$ such that two consecutive cycles $C_{i}$ and $C_{i+1}$ share a single
edge, this edge being $u_{i}v_{i}$ for $1\leq i\leq k-1.$ Further, the length
of cycles $C_{1}$ and $C_{k}$ is $a_{1}+2$ and $a_{k}+2,$ respectively. The
length of a cycle $C_{i},$ for $2\leq i\leq k-1,$ is $a_{i}+b_{i}+2.$ Also,
observe that $G(S)$ has
\[
a_{1}+a_{k}+2(k-1)+\sum_{i=2}^{k-1}\big((a_{i}-1)+(b_{i}-1)\big)=\sum
_{i=1}^{k-1}(a_{i}+b_{i+1})+2
\]
vertices.

It is worthy to note that this family of graphs includes the so called
unbranched catacondensed benzenoids, this is the case when all cycles in the
sequence are of the length six. Hence, our results have a possible application
in chemistry, as they can be used to derive chemical properties of the
mentioned chemical compounds, see for example \cite{Anstoter2020} and all the
references within.

\bigskip

We have the following statement.

\begin{theorem}
\label{thm:sequence}Let $k\geq2$ and let $S=(a_{1},a_{2},\dots,a_{k}%
;b_{1},b_{2},\dots,b_{k})$ be a sequence of positive integers such that
$a_{1}=b_{1}$ and $a_{k}=b_{k}$. Denote $n=\sum_{i=1}^{k-1}(a_{i}+b_{i+1})+2$.
Then
\begin{align*}
\mathrm{pn}(G(S))={}  &  \sum_{1\leq i<j\leq k}(a_{i}b_{j}+b_{i}%
a_{j})2^{j-i-1}(i+1)(k-j+2)\ -\ a_{1}a_{k}2^{k}\\
&  +\sum_{2\leq i<j\leq k-1}(a_{i}a_{j}+b_{i}b_{j})2^{j-i-1}(i+1)(k-j+2)\\
&  +\sum_{i=2}^{k-1}\bigg(\binom{a_{i}+1}{2}+\binom{b_{i}+1}{2}%
\bigg)(1+i(k-i+1))\\
&  +\bigg(\binom{a_{1}+2}{2}+\binom{a_{k}+2}{2}\bigg)(k+1)\\
&  +\sum_{i=2}^{k-1}(a_{i}+1)(b_{i}+1)(k+1)\ -\ (k-1)(k+1)+n.
\end{align*}

\end{theorem}

\begin{proof}
First, denote the set of $a_{1}$ (resp. $a_{k}$) interior vertices of a path
connecting $u_{1}$ with $v_{1}$ (resp. $u_{k-1}$ with $v_{k-1}$) by $A_{1}$
(resp. $A_{k}$) and let $B_{1}=A_{1}$ and $B_{k}=A_{k}$. Further, denote the
set of $a_{i}-1$ (resp. $b_{i}-1$) interior vertices of a path connecting
$u_{i-1}$ with $u_{i}$ (resp. $v_{i-1}$ with $v_{i}$) by $A_{i}$ (resp.
$B_{i}$), where $2\leq i\leq k-1$. Also, recall that $C_{i}$ denotes the cycle
consisting of $u_{i-1},v_{i-1},u_{i},v_{i},A_{i},B_{i}$, for $2\leq i\leq
k-1$, and that $C_{1}$ (resp. $C_{k}$) denotes the cycle consisting of
$u_{1}v_{1},A_{1}$ (resp. $u_{k-1},v_{k-1},A_{k}$). To simplify the
argumentation, we assume that edges $u_{1}v_{1},u_{2}v_{2},\dots
,u_{k-1}v_{k-1}$ are drawn in a vertical position and the cycles $C_{1}%
,C_{2},\dots,C_{k}$ are convex polygons. For each pair of distinct vertices
$x$ and $y$ of $G(S)$ we count the number of paths between $x$ and $y.$

\medskip

\noindent\textbf{Case 1:} $x\in\{u_{i-1}\}\cup A_{i}$\emph{ and }$y\in
B_{j}\cup\{y_{j}\}$\emph{, where }$1\leq i<j\leq k$\emph{.} Consider a path
$P$ between $x$ and $y.$ If we cut $G(S)$ vertically between $x$ and $y$, the
cut must intersect $P$ odd number of times, hence once. On the other hand if
we cut $G(S)$ vertically on the left-hand side of $x$ or on the right-hand
side of $y$, then the cut intersects $P$ twice or $0$ times. Thus, if $P$
contains $u_{\ell}v_{\ell}$ for $\ell<i$, this edge determines $P$ on the
left-hand side from $x$. Analogously, if $P$ contains $u_{\ell}v_{\ell}$ for
$\ell\geq j$, this edge determines $P$ on the right-hand side from $y$. But,
between $x$ and $y$ any choice of edges $u_{i}v_{i},u_{i+1}v_{i+1}%
,\dots,u_{j-1}v_{j-1}$ is possible since each such choice can be completed to
a path in $G(S)$ (in fact to $2$ distinct paths in $G(S)$) in the vertical
strip of plane between $x$ and $y$.

Denote $E_{i,j}=\{u_{i}v_{i},u_{i+1}v_{i+1},\dots,u_{j-1}v_{j-1}\}$ and let
$T\subseteq E_{i,j}.$ Let us count the paths $P$ connecting $x$ and $y$ such
that $E(P)\cap E_{i,j}=T.$ Assume first that $T$ contains an even number of
edges of $E_{i,j}.$ If the path $P$ starts by going right from $x,$ then there
are $k-j+1$ such paths. On the other hand, if the path $P$ starts from $x$ by
going left or down, then there are $i$ such paths. Since there are
$2^{(j-1)-(i-1)}=2^{j-i}$ subsets $T$ of $E_{i,j},$ half of which are even
sized, there are
\[
2^{j-i-1}((k-j+1)+i)
\]
paths between $x$ and $y$ which contain an even sized $T\subseteq E_{i,j}.$

Assume next that $T$ contains an odd number of edges of $E_{i,j}.$ If the path
$P$ starts from $x$ by going right, there is only one path between $x$ and $y$
containing $T$. On the other hand, if $P$ starts from $x$ by going left or
down, there are $i(k-j+1)$ such paths between $x$ and $y.$ Again, since there
are $2^{j-i-1}$ odd sized subsets $T$ of $E_{i,j},$ we conclude that there
are
\[
2^{j-i-1}(1+i(k-j+1))
\]
paths between $x$ and $y$ containing odd sized $T\subseteq E_{i,j}.$

Summing the above two numbers yields
\[
2^{j-i-1}((k-j+1)+i+1+i(k-j+1))=2^{j-i-1}(i+1)(k-j+2).
\]
which is the total number of subpaths of $G(S)$ connecting $x$ and $y.$

\medskip

\noindent\textbf{Case 2:} $x\in\{v_{i-1}\}\cup B_{i}$\emph{ and }$y\in
A_{j}\cup\{u_{j}\}$\emph{, where }$1\leq i<j\leq k$\emph{.} By the analogous
reasoning as in Case 1, we conclude that there are $2^{j-i-1}(i+1)(k-j+2)$
subpaths of $G(S)$ connecting $x$ and $y$.

\medskip

Notice that there are $a_{i}b_{j}+b_{i}a_{j}$ pairs of vertices $x$ and $y$
which belong to Cases 1 or 2. Also, the case $i=1$ and $j=k$ is included in
both cases, so in the final sum we have to subtract $a_{1}b_{k}2^{k-2}%
(1+1)(k-k+2)=a_{1}a_{k}2^{k}$ from these two sums. Hence, the total number of
subpaths of $G(S)$ which connect pairs of vertices from Cases 1 or 2 equals
\[
\sum_{1\leq i<j\leq k}(a_{i}b_{j}+b_{i}a_{j})2^{j-i-1}(i+1)(k-j+2)\ -\ a_{1}%
a_{k}2^{k}%
\]
which is the expression from the first line of the formula from the statement
of Theorem~{\ref{thm:sequence}}.

\medskip

\noindent\textbf{Case 3:} $x\in\{u_{i-1}\}\cup A_{i}$\emph{ and }$y\in
A_{j}\cup\{u_{j}\}$\emph{, where }$2\leq i<j\leq k-1$\emph{.} Observe that the
cases when $x\in A_{1}$ or $y\in A_{k}$ were counted already in the previous
two cases, so they are not included here. Here, the role of even and odd
subsets $T\subseteq E_{i,j}$ is reversed. Namely, for an even sized $T,$ there
is $1$ path $P$ starting from $x$ to the right and ending in $y,$ and there
are $i(k-j+1)$ paths $P$ starting from $x$ to the left or down and ending in
$y.$ On the other hand, if $T$ is odd sized, then there is $k-j+1$ paths
starting from $x$ to the right, and $i$ paths starting from $x$ to the left or
down. Since the number of both even sized and odd sized subsets $T\subseteq
E_{i,j}$ is equal to $2^{j-i-1},$ we obtain that there are precisely%
\[
2^{j-i-1}(1+i(k-j+1)+(k-j+1)+i)=2^{j-i-1}(i+1)(k-j+2)
\]
subpaths of $G(S)$ which connect $x$ and $y.$ Notice that this number is the
same as in Cases 1 and 2.

\medskip

\noindent\textbf{Case 4:} $x\in\{v_{i-1}\}\cup B_{i}$\emph{ and }$y\in
B_{j}\cup\{v_{j}\}$\emph{, where }$2\leq i<j\leq k-1$\emph{.} By the analogous
reasoning as in Case 3, we obtain the same number of subpaths between $x$ and
$y.$

\medskip

Since there are $a_{i}a_{j}$ (resp. $b_{i}b_{j}$) pairs of vertices $x$ and
$y$ which belong to Case 3 (resp. Case 4), these pairs contribute to
$\mathrm{pn}(G(S))$ by the sum from the second line of the formula from the
statement of this theorem.

\medskip

\noindent\textbf{Case 5:} $x,y\in V(C_{i})$\emph{, where }$1\leq i\leq
k$\emph{.} Here, we distinguish several subcases, the one is when both $x$ and
$y$ belong to $A_{i}\cup\{u_{i-1},u_{i}\},$ the other is when both $x$ and $y$
belong to $B_{i}\cup\{v_{i-1},v_{i}\}$, and the last is when $x$ belongs to
the former and $y$ to the latter set. The case when $i=1$ and $i=k$ has to be
considered separately.

Let us first assume that $x,y\in A_{i}\cup\{u_{i-1},u_{i}\}$, where $2\leq
i\leq k-1$ and $x\neq y$. We may assume that $x$ is on the left-hand side from
$y$. If $P$ goes from $x$ to the right towards $y$, then we have just $1$
possibility. But going from $x$ to the left-hand side we have $i$
possibilities at the beginning and $k-j+1$ possibilities at the end, which
gives $1+i(k-j+1)$ distinct subpaths of $G(S)$ connecting $x$ and $y.$ Since
there are $\binom{a_{i}+1}{2}$ pairs of vertices $x$ and $y$ with $x,y\in
A_{i}\cup\{u_{i-1},u_{i}\},$ such pairs contribute to $\mathrm{pn}(G(S))$
with
\[
\sum_{i=2}^{k-1}\binom{a_{i}+1}{2}\big(1+i(k-j+1)\big).
\]

Assume next that $x,y\in B_{i}\cup\{v_{i-1},v_{i}\}$, where $2\leq i\leq k-1$
and $x\neq y$. The same reasoning as above yields the analogous sum, hence we
obtain the sum in the third line of the of the formula from the statement of
this theorem.

Let us next consider the cases $i=1$ and $i=k$. Observe that there are two
pats connecting $x$ with $y$ which contain none of $u_{1}v_{1},u_{2}%
,v_{2},\dots,u_{k-1}v_{k-1}$ and for each $u_{i}v_{i}$, $1\le i\le k-1$, there
is a unique path connecting $x$ with $y$ which contains $u_{i}v_{i}$. Hence,
all pairs $x,y\in V(C_{1})$ contribute to $G(S)$ by $\binom{a_{1}+2}{2}(k+1).$
The same reasoning yields analogous sum for $x,y\in V(C_{k}),$ which gives the
expression in the fourth line of the formula from the statement of this theorem.

Finally, assume that $x\in A_{i}\cup\{u_{i-1},u_{i}\}$ and $y\in B_{i}%
\cup\{v_{i-1},v_{i}\}$, where $2\leq i\leq k-1$. There are $(a_{i}%
+1)(b_{i}+1)$ such pairs $x$ and $y,$ and notice that each path connecting $x$
and $y$ contains at most one vertical edge of $G(S)$. So there are $k+1$ paths
connecting $x$ with $y$, two of which contain no edge from $u_{1}v_{1}%
,u_{2},v_{2},\dots,u_{k-1}v_{k-1}$. The problem is that paths connecting
$u_{j}$ with $v_{j}$ are counted twice here if $2\leq j\leq k-2$, and paths
connecting $u_{1}$ with $v_{1}$ and those connecting $u_{k-1}$ with $v_{k-1}$
are already included in $\big(\binom{a_{1}+2}{2}+\binom{a_{k}+2}{2}%
\big)(k+1)$. Therefore, we must subtract their number once. Since there are
$k-1$ such pairs each connected by $k+1$ subpaths, this number is
$(k-1)(k+1).$ Finally, a number of subpaths of zero length which is equal to
the number $n$ of vertices in $G(S)$ must be included into $\mathrm{pn}%
(G(S)),$ which gives the last line of the formula from the statement of this theorem.
\end{proof}

In a cycle chain $G(S),$ let us denote by $g_{i}$ the length of cycle $C_{i}$
for $1\leq i\leq k.$ Notice that this implies $g_{1}=a_{1}+2,$ $g_{k}%
=a_{k}+2,$ and $g_{i}=a_{i}+b_{i}+2$ for $2\leq i\leq k-1.$ A cycle $C_{i}$ of
$G(S)$ is called \emph{interior} if $2\leq i\leq k-1$. An interior cycle
$C_{i}$ is \emph{linear} (resp. \emph{almost linear}) if $a_{i}=b_{i}$ (resp.
$\left\vert a_{i}-b_{i}\right\vert =1$). Notice that an interior cycle of an
even length can be linear but not almost linear, and for an odd length cycle
it is vice versa. Further, an interior cycle $C_{i}$ is called a \emph{kink}
if $a_{i}=1$ or $b_{i}=1.$ A chain $G(S)$ is a \emph{kink} chain if every
interior cycle is a kink. Similarly, $G(S)$ is \emph{linear} if every interior
cycle of $G(S)$ is linear. If $G(S)$ contains an odd length cycle, then it
cannot be linear, so we additionally introduce the notion of an "almost"
linear chain. A chain $G(S)$ is \emph{almost linear} if it is not linear and
each of its interior cycles is either linear or almost linear.

\begin{corollary}
\label{Cor_chains}Let $(g_{1},\ldots,g_{k})$ be a sequence of integers with
$g_{i}\geq4$ for every $i.$ Let $\mathcal{G}(g_{1},\ldots,g_{k})$ be the
family of cycle chains with $k$ cycles $C_{i}$ of length $g_{i}$ where $k
\ge3$ and $1\leq i\leq k.$ The maximum of the subpath number in the family
$\mathcal{G}(g_{1},\ldots,g_{k})$ is attained only for a kink chain, and the
minimum only for a linear or an almost linear chain.
\end{corollary}

\begin{proof}
We consider the formula for $\mathrm{pn}(G(S))$ from Theorem
\ref{thm:sequence}. First, notice that $a_{1}$ and $a_{k}$ are the same for
all chains in the family $\mathcal{G}(g_{1},\ldots,g_{k}),$ hence they do not
contribute to the difference of the subpath number for different chains in the
family. Next, denote $p_{i,j}=a_{i}b_{j}+b_{i}a_{j}$ and $q_{i,j}=a_{i}%
a_{j}+b_{i}b_{j}.$ It is easily verified that $p_{i,j}+q_{i,j}=\left(
a_{i}+b_{i}\right)  \left(  a_{j}+b_{j}\right)  .$ Hence, for $2\leq i\leq
k-1$ we have $p_{i,j}+q_{i,j}=(g_{i}-2)(g_{j}-2).$ Also, since $a_{1}=b_{1},$
we have $p_{1,j}=a_{1}(g_{j}-2)$ for $2\leq j\leq k-1.$ Similarly,
$p_{i,k}=(g_{i}-2)a_{k}$ for $2\leq i\leq k-1.$ Finally, $p_{1,k}=2a_{1}%
a_{k}.$ This implies that the expressions from the first two lines of the
formula are the same for all chains in the considered family.

The expression from the fourth line is also the same for all chains in the
family, so let us consider the expressions from the third and fifth line of
the formula. Denote $r_{i}=\binom{a_{i}+1}{2}+\binom{b_{i}+1}{2}$ and
$s_{i}=(a_{i}+1)(b_{i}+1).$ For $2\leq i\leq k-1,$ it is easily verified that
$r_{i}+s_{i}=g_{i}(g_{i}-1)/2.$ Hence, the difference of the subpath number
for a pair of chains in the family $\mathcal{G}(g_{1},\ldots,g_{k})$ is
determined by the expression%
\[
\sum_{i=2}^{k-1}r_{i}(1+i(k-i+1)-(k+1))=\sum_{i=2}^{k-1}r_{i}(-i^{2}%
+i(k+1)-k).
\]
It is easily seen that $-i^{2}+i(k+1)-k$ is strictly positive for $2\leq i\leq
k-1.$ Hence, the subpath number is maximized (resp. minimized) by the chain
with maximum (resp. minimum) possible $r_{i}$ for every $2\leq i\leq k-1.$
Notice that $g_{i}=a_{i}+b_{i}+2$ implies%
\[
r_{i}=\binom{a_{i}+1}{2}+\binom{g_{i}-a_{i}-1}{2}.
\]
Now, it is easily verified that $r_{i}$ is minimum for $a_{i}=\left\lfloor
g_{i}/2\right\rfloor -1$ or $a_{i}=\left\lceil g_{i}/2\right\rceil -1.$ And
$r_{i}$ is maximum for $a_{i}$ minimum or maximum possible which is equivalent
with $a_{i}=1$ or $b_{i}=1.$
\end{proof}

\begin{figure}[h]
\begin{center}
\includegraphics[scale=0.5]{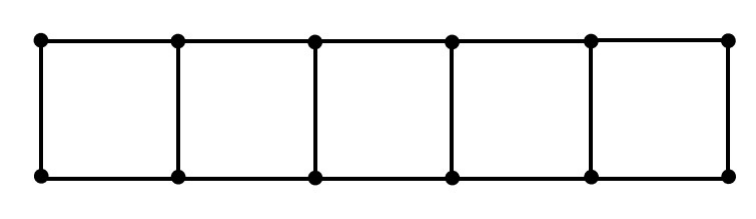}
\end{center}
\caption{The figure shows the only $4$-chain with $k=5.$}%
\label{Fig_rectangle}%
\end{figure}

Let us next show that the results of Theorem \ref{thm:sequence} and Corollary
\ref{Cor_chains} can be applied to various families of chains frequently
studied in the literature. For an integer $g\geq4,$ a $g$\emph{-chain} is a
cycle chain in which every cycle is of the length $g.$ For $g=4$ there exists
only one $4$-chain on a given number of cycles $k.$ This chain is illustrated
by Figure \ref{Fig_rectangle}. The $4$-chain is also called the \emph{ladder
graph}, and a cycle of length $4$ in a ladder graph will be called a
\emph{square}. Using Theorem \ref{thm:sequence} the subpath number of the
ladder graph is easily established.

\begin{corollary}
For the ladder graph $G$ with $k$ squares we have
\[
\mathrm{pn}(G)=36\cdot2^{k}-\frac{1}{3}k^{3}-4k^{2}-\frac{56}{3}k-33.
\]

\end{corollary}

The next interesting value of $g$ is $g=6.$ Cycle chains with $g=6$ are also
called \emph{hexagonal chains} and the cycles of length $6$ are called
\emph{hexagons}. This family of chains has applications in chemistry, since
such graphs model unbranched catacondesed benzenoids (see \cite{Anstoter2020}
for example). For $k\leq2$ there is only one hexagonal chain with $k$
hexagons. For $k \geq3 $, there exist multiple such chains, raising the
natural question of which chain maximizes or minimizes a particular quantity.
This question was investigated for many graph quantities, among them for the
Wiener index \cite{GutmanHexagonal} and the subtree number
\cite{YuYangHexagonal}. In particular, it is established that for a hexagonal
chain $G$ with $k$ hexagons, it holds that $W(H_{k})<W(G)<W(L_{k})$ and
$N(L_{k})<N(G)<N(H_{k}),$ provided that $G\not \in \{L_{k},H_{k}\}.$

\begin{figure}[h]
\begin{center}%
\begin{tabular}
[t]{cccc}%
\multicolumn{4}{c}{%
\begin{tabular}
[t]{ll}%
a) & \raisebox{-0.9\height}{\includegraphics[scale=0.5]{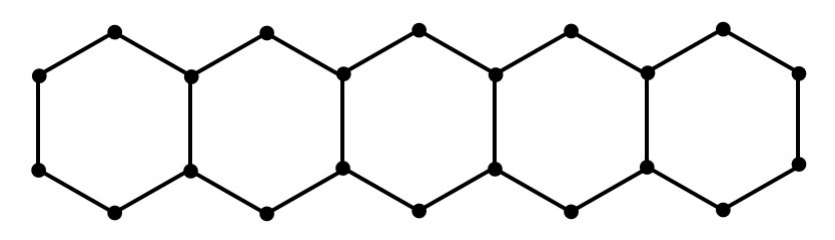}}
\end{tabular}
}\\
b) & \raisebox{-0.9\height}{\includegraphics[scale=0.5]{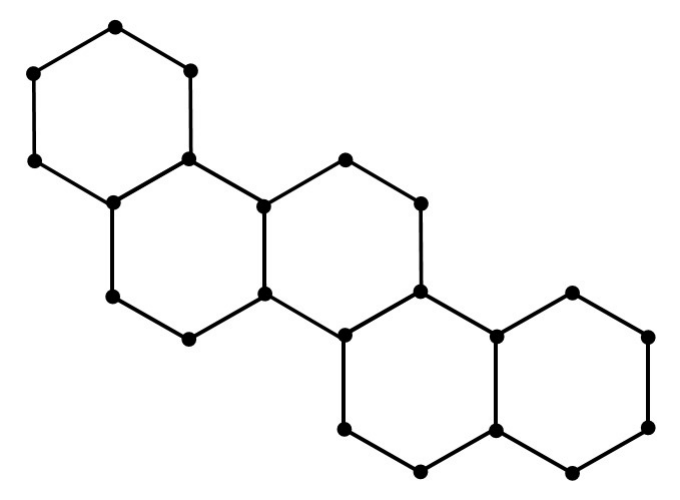}} & c) &
\raisebox{-0.9\height}{\includegraphics[scale=0.5]{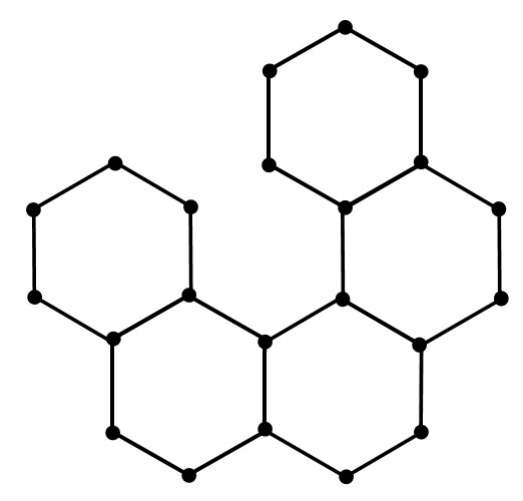}}
\end{tabular}
\end{center}
\caption{The figure shows hexagonal chains: a) the linear chain $L_{5},$ b)
the zig-zag chain $Z_{5}$ and c) the helicene $H_{5}.$}%
\label{Fig_hexagon}%
\end{figure}

Regarding the same question of extremal hexagonal chains with respect to the
subpath number, Corollary \ref{Cor_chains} provides the answer. \ Namely,
notice that the linear hexagonal chain with $k$ cycles is unique, denote it by
$L_{k}.$ On the other hand, there are many kink hexagonal chains with $k$
hexagons, for example the zig-zag chain $Z_{k}$ in which $a_{2i}=1$ and
$a_{2i+1}=3$ for $2\leq2i<2i+1\leq k-1$ or the helicene $H_{k}$ in which
$a_{i}=1$ for $2\leq i\leq k-1.$ These chains are illustrated by Figure
\ref{Fig_hexagon}. Corollary \ref{Cor_chains} now yields the following result.

\begin{corollary}
Let $G$ be a hexagonal chain with $k$ hexagons. Then%
\[
144\cdot2^{k}-\frac{5}{3}k^{3}-21k^{2}-\frac{265}{3}k-141\leq\mathrm{pn}%
(G)\leq144\cdot2^{k}-\frac{3}{2}k^{3}-\frac{43}{2}k^{2}-88k-141.
\]
The lower bound is attained if and only if $G=L_{k}$, and upper bound is
attained only by a kink chain $G.$
\end{corollary}

Notice that the subpath number differs from the subtree number in the sense
that any kink chain maximizes it, and not only the helicene $H_{k}.$ For
example, the zig-zag chain $Z_{k}$ also maximizes the subpath number among
chains with $k$ hexagons.

\section{Concluding remarks and further work}

As a direction for further work, it would be interesting to investigate the
subpath number for specific graph classes. Lemma \ref{Lemma_aditivity} implies
that removing edges from a graph decreases the subpath number, and that adding
edges increases the subpath number. These two processes lead to a graph which
contains a vertex of degree $1$ or vertex of maximum possible degree,
respectively. Hence, the result of Theorem \ref{Tm_generalG} that trees are
minimal and the complete graph maximal with respect to the subpath number is
implied. An interesting question that arises is what happens when the degree
is prescribed?

The most simple case of prescribed degrees is a $r$-regular graph $G$ in which
all vertices have the same degree $r.$ If $r=2,$ then $G$ is the cycle for
which the subpath number is established. So, the first non-trivial case is
$r=3,$ i.e. cubic graphs. Let us define graphs for which we believe are
extremal in this family.

\begin{figure}[h]
\begin{center}%
\begin{tabular}
[t]{ll}%
a) & \raisebox{-0.9\height}{\includegraphics[scale=0.5]{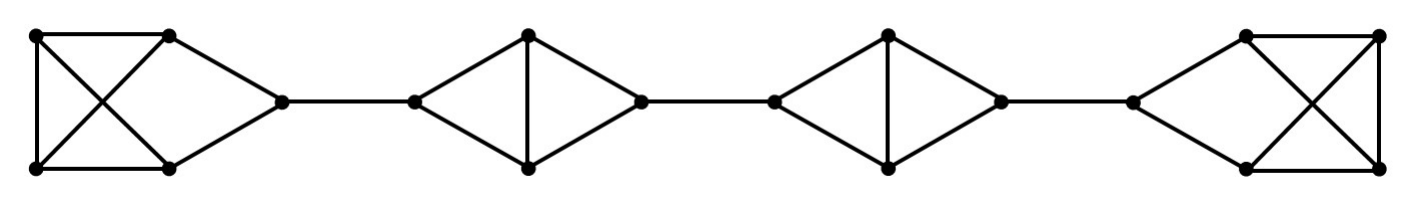}}\\
b) & \raisebox{-0.9\height}{\includegraphics[scale=0.49]{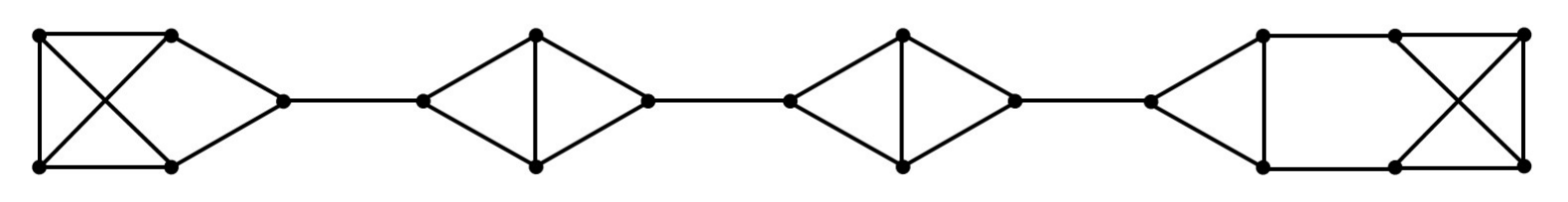}}
\end{tabular}
\end{center}
\caption{The graphs: a) $L_{18},$ b) $L_{20}.$}%
\label{Fig_Ln}%
\end{figure}

Assume $n\geq10$ is even. If $n=4q-2$ for some integer $q,$ then the graph
$L_{n}$ is obtained from $(n-10)/4$ copies of $K_{4}-e$ by adding $(n-10)/4-1$
edges so that the resulting graph is connected and has maximum degree 3. Thus,
we obtain a path like structure which is then ended on both sides by a pendant
block, each on $5$ vertices. Otherwise, if $n=4q,$ then $L_{n}$ is obtained
from $(n-12)/4$ copies of $K_{4}-e$ connected into a path like structure in
the same way as before, but now the structure is ended with two pendant
blocks, one on $5$ vertices and the other on $7$ vertices. The graphs $L_{n}$
are illustrated by Figure \ref{Fig_Ln}. We have the following conjecture.

\begin{conjecture}
\label{Con_minCubic}In the class of cubic graphs on $n$ vertices, the graph
$L_{n}$ is the only graph which minimizes the subpath number.
\end{conjecture}

The cubic graphs which maximize the subpath number seem to be more elusive.
Interestingly, for $n=8$ it is the M\"{o}bius strip and for $n=10$ it is the
Petersen graph. For other small values of $n\geq12,$ the maximal graphs are
rather symmetrical, though not always transitive, and they tend to have a
relatively small diameter, though not always the smallest possible. Hence, we
propose the following problem.

\begin{problem}
\label{Con_maxCubic}Find the graphs which maximize the subpath number in the
class of cubic graphs on $n$ vertices.
\end{problem}

Finally, based on our investigation of bipartite graphs, where the complete
bipartite graph $K_{\left\lceil n/2\right\rceil ,\left\lfloor n/2\right\rfloor
}$ maximizes the subpath number, we believe that the same graph maximizes the
subpath number in a wider family of graphs also, namely the triangle-free
graphs. We state it as the following conjecture.

\begin{conjecture}
Among triangle-free graphs on $n$ vertices the maximum value of the subpath
number is attained only by the complete bipartite graph $K_{\left\lceil
n/2\right\rceil ,\left\lfloor n/2\right\rfloor }.$
\end{conjecture}

\bigskip

\vskip1pc \noindent\textbf{Acknowledgments.}~~ This work is partially
supported by Slovak research grants VEGA 1/0069/23, VEGA 1/0011/25,
APVV-22-0005 and APVV-23-0076, by Slovenian Research and Innovation Agency
ARIS program\ P1-0383, project J1-3002 and the annual work program of
Rudolfovo, by Project KK.01.1.1.02.0027 co-financed by the Croatian Government
and the European Union through the European Regional Development Fund - the
Competitiveness and Cohesion Operational Programme, by bilateral
Croatian-Slovenian project BI-HR/25-27-004 and by the Key Scientific and
Technological Project of Henan Province, China (grant nos. 242102521023), the
Youth Project of Humanities and Social Sciences Research of the Ministry of
Education (No. 24YJCZH199).



\end{document}